\documentclass[12pt,reqno,a4paper]{amsart}
\usepackage{amssymb}

\usepackage{amsfonts,amsmath,amsthm}
\usepackage{mathrsfs}
\usepackage{dsfont}
\usepackage{color}
\usepackage{hyperref}
\usepackage[foot]{amsaddr}

\makeatletter
\@namedef{subjclassname@2020}{%
  \textup{2020} Mathematics Subject Classification}
\makeatother

\usepackage[T1]{fontenc}

\newcommand\ran{\operatorname{ran}}

\newcommand{\rem}[1]{}

\newtheorem{theorem}{Theorem}[section]
\newtheorem{corollary}[theorem]{Corollary}
\newtheorem{lemma}[theorem]{Lemma}
\newtheorem{proposition}[theorem]{Proposition}
\theoremstyle{definition}
\newtheorem{definition}[theorem]{Definition}

\newtheorem{example}[theorem]{Example}
\newtheorem{remark}[theorem]{Remark}

\renewcommand\leq{\leqslant}
\renewcommand\geq{\geqslant}

\def\dd{\mathrm{d}}

\def\NN{\mathbb N}
\def\RR{\mathbb R}

\newcommand{\cL}{\mathcal{L}}
\newcommand{\cS}{\mathbf{S}}

\newcommand{\rL}{\mathrm{L}}
\newcommand{\rC}{\mathrm{C}}

\newcommand{\rF}{\mathrm{F}}
\newcommand{\rReg}{\mathrm{Reg}}

\def\fs#1{\textcolor[rgb]{0.0,0.0,0.0}{#1}}

\title{A refinement of Baillon's theorem on maximal regularity}

\author[B.~Jacob]{Birgit Jacob}
\address[BJ, JW]{University of Wuppertal, School of Mathematics and Natural Sciences, IMACM, Gau\ss stra\ss e 20, D-42119 Wuppertal, Germany.}
 \email{bjacob@uni-wuppertal.de,wintermayr@uni-wuppertal.de}   
\author[F.L.~Schwenninger]{Felix L.~Schwenninger}
\address[FLS]{Department of Applied Mathematics, University of Twente, P.O.~Box 217, 
7500 AE Enschede, 
The Netherlands and \newline Department of Mathematics, Center for Optimization and Approximation, University of Hamburg, 
Bundesstr. 55, 20146 Hamburg, Germany}
\email[corresponding author]{f.l.schwenninger@utwente.nl}

\author[J.~Wintermayr]{Jens Wintermayr}

\begin{document}

\begin{abstract}
		By Baillon's result, it is known that maximal regularity with respect to the space of continuous functions is rare; 
		it implies that either the involved semigroup generator is a bounded operator or the considered space contains $c_{0}$.
		We show that the latter alternative can be excluded under a refined condition resembling maximal regularity with respect to $\mathrm{L}^{\infty}$.
\end{abstract}
	

	
\maketitle

\numberwithin{equation}{section}
\section{Introduction}
The question whether the solutions to an abstract Cauchy problem 
	\begin{equation}\label{acp}
	\arraycolsep=1pt\def\arraystretch{1.5}
		\begin{array}{ccl}\frac{\mathrm{d}}{\mathrm{d}t}x(t)&=&Ax(t)+f(t),\qquad t\in [0,\tau], \\x(0)&=&0,\end{array}
	\end{equation}
	where $A$ generates a strongly continuous semigroup $\cS=(S(t))_{t\ge0}$ on a Banach space $X$, 
preserve the regularity of the inhomogeneity $f:[0,\tau]\to X$ is omnipresent in the study of parabolic equations. They turn out to be particularly useful for investigating nonlinear equations, see e.g.\ \cite{Aman95,DenkHiebPrue03,KunstWeis04,lunardi} and the references therein. More precisely, \emph{maximal regularity} of the semigroup (or, equivalently, the generator) requires that  $\frac{d}{dt}x$ and $Ax$ have the same regularity as $f$, e.g.\ that $\frac{d}{dt}x$ and $Ax$ are well-defined in $\rL^{p}((0,\tau),X)$ for any $f\in \rL^{p}((0,\tau),X)$, $\tau>0$, where $x:[0,\tau]\to X$ refers to the mild solution to \eqref{acp}. This property is equivalent to an inequality of the form
\begin{equation}\label{maxreg1}
\left\|A\int_{0}^{\cdot}S(\cdot-s)f(s)\mathrm{d}s\right\|_{\rL^{p}((0,\tau),X)} \leq \kappa_{\tau}\|f\|_{\rL^{p}((0,\tau),X)},
\end{equation}
for some constant $\kappa_{\tau}>0$ and all $f\in \rL^p((0,\tau),X)$. The theory on maximal regularity has started with the works by de Simon and Sobolevskii \cite{deSi64,Sobo64}, who showed that analyticity of the semigroup is a sufficient condition on Hilbert spaces when $p\in(1,\infty)$. In fact, also on general Banach spaces, analyticity is necessary for maximal regularity and the property is independent of the particular choice $p\in(1,\infty)$, \cite{CannVesp86,CoulLamb86,Dore00}, see also \cite{CronSimo11} for the case of continuous functions. However, this characterization fails to be true for general non-Hilbert spaces, as was shown by Kalton--Lancien \cite{KaltLanc00}, see also \cite{Fack13,Fack14}. The appropriate replacement for UMD spaces was found by Weis to be the property that $A$ is \emph{$R$-sectorial}, \cite{Weis01}. 

On the other hand, the cases $p=1$ and $p=\infty$  are exotic in a way. In 1980, Baillon \cite{Bail80} proved that maximal regularity with respect to the space of the continuous functions---i.e.\ replacing ``$\rL^{p}((0,\tau),X)$'' by ``$\rC([0,\tau],X)$'' in \eqref{maxreg1}---implies that $A$ must be bounded when $X$ does not contain an isomorphic copy of the sequence space $c_{0}$. A rather simple example, due to T.~Kato, of an unbounded operator $A$ on $c_{0}$---which had been known prior to Baillon's work---shows that the latter assumption on $X$ cannot be dropped in general. Note that a simplified proof of Baillon's result can be found in \cite{EG} and that the case $X=\rL^{2}$ had even been treated earlier in \cite{DaPrGris79}. Moreover, it is not hard to see that in Baillon's result ``$\rC$-maximal regularity'' may be replaced by ``$\rL^{\infty}$-maximal regularity'', see also \cite{GuerDela95}. The dual situation of $\rL^{1}$-maximal regularity was covered by Guerre-Delabri\`ere \cite{GuerDela95}.
In \cite{Travis}, see also the comments in \cite{ChyaShawPisk99}, Travis (implicitly) showed that $\rC$-maximal regularity is equivalent to the property that $\cS$ is of  \emph{bounded semivariation} on some interval $[0,\tau]$, i.e., 
\begin{equation}
\label{eq:SV}
 \operatorname{var}_0^\tau(\cS):= \sup_{\|x_{i}\|\leq 1,\; 0={}t_{1}<t_{2}<\dots<t_{n}=\tau, \; n\in\mathbb{N}}\left\|\sum_{i=1}^{n-1}(S(t_{i})-S(t_{i+1}))x_{i}\right\|<\infty,
 \end{equation}
and, moreover, that this is equivalent to the property that every weak solution of \eqref{acp} is in fact a classical solution. Travis' proof yet reveals another characterization of $\rC$-maximal regularity; namely that
\[\Phi_{\tau}: \rC([0,\tau],X)\to X, f\mapsto A\int_{0}^{\tau}S(\tau-s)f(s)\mathrm{d}s\]
is a well-defined bounded operator. In the present work we investigate the consequence of sharpening the condition of $\rC$-maximal regularity by refining this latter property. More precisely, we study the assumption that $\Phi_{\tau}$ is even bounded from $\rL^{\infty}((0,\tau),X)$ to $X$, which---as should be emphasized--- is stronger than plain $\rL^{\infty}$-maximal regularity. Indeed, Kato's example satisfies $\rL^{\infty}$-maximal regularity, but fails to have the refined property of $\Phi_{\tau}$ being bounded from $\rL^{\infty}((0,\tau),X)$ to $X$, see Example \ref{ex:Kato} below. This is not accidental; our main result, Theorem \ref{thm:main}, states that this condition always implies that $A$ is bounded, independent of whether $X$ contains $c_{0}$ or not. Therefore, our contribution can be seen as a refinement of Baillon's result. Moreover, since any bounded generator is easily seen to imply that $\Phi_{\tau}\in\mathcal{L}(\rL^{\infty}((0,\tau),X),X)$, the result establishes yet another characterization of when a strongly continuous semigroups is in fact uniformly continuous, or, equivalently, when the generator $A$ is a bounded operator. Results of this type often rest on dichotomy laws on norm bounds for $S(t)-I$ or $tAS(t)$ as $t\to0^{+}$, see e.g.\ Lotz \cite{BookPosSg86,Lotz85}, Hille \cite{Hill50} or \cite{BerkEstMokh03}. A related argument indeed shows that in order to prove our main theorem, it would suffice to show that the norm of $\Phi_{\tau}$ tends to $0$ as $\tau\to0^{+}$, Proposition \ref{Aboundedregu}. Note that also here the crucial difference between bounds of $\Phi_{\tau}$ on $\rC([0,\tau],X)$ and $\rL^{\infty}((0,\tau),X)$ becomes apparent: Whereas in the former case Kato's example on $X=c_{0}$ shows that $\limsup_{\tau\to0^{+}}\|\Phi_{\tau}\|>0$, an analogous construction for a potential counterexample in the latter case would require $X=\ell^{\infty}$. But then, by Lotz's result \cite{Lotz85}, $A$ was bounded trivially. Unfortunately, this sufficient (and necessary) condition on norm convergence of $\Phi_{\tau}$ seems to be hard to access. Therefore, our proof approach follows a different path, which exploits both Baillon's result as well as information on the Favard space of the semigroup, Theorem \ref{thm:Fav}. These ingredients are then combined by using a remarkable characterization of isomorphisms $R:c_{0}\to Y$ due to Lotz--Peck--Porta \cite{LotzPeckPort79} and Bourgain--Rosenthal \cite{BourRose83}, which was previously applied by van Neerven \cite[Theorem 3.2.10]{Nee} to characterize uniformly continuous semigroups on $c_{0}$ by properties of the respective Favard spaces. 
\\ It is natural to investigate the dual condition of $\Phi_{\tau}$ being bounded from $\rL^{\infty}$ to $X$, which reduces to boundedness of the mapping
\[ \Psi_{\tau}:X\to \rL^{1}((0,\tau),X),\fs{x\mapsto AS(\cdot)x},\]
which, a-priori, is well-defined on the domain of $A$. However, by means of examples it is not hard to see that this condition, \fs{which is in fact equivalent to $\rL^{1}$-maximal regularity, see \cite[Theorem 3.6]{KaltPort08} and Proposition \ref{propMR1a}}, does not imply boundedness of $A$ in general. \fs{Nevertheless, $\rL^1$-maximal regularity pertains to the study of existence and stability of nonlinear equations such as in  mathematical fluid flow \cite{DancHiebMuchTolk20,MuchDanc09,MuchDanc15}. However, there the space $X$  in \eqref{maxreg1} is replaced by some interpolation space between $X$ and the domain of $A$, which is common for results on maximal regularity. We also refer to \cite{RiFarw20}, where such an abstract $\rL^{1}$-case is in turn deduced from a classical $\rC$-maximal regularity result due to Da Prato--Grisvard \cite{DaPrGris79}.}

In Section \ref{sec:discrete} we discuss the discrete-time analogs of the previously derived results and show that the difference in maximal regularity and its refined notion vanishes in this situation, which makes this case less interesting. 
In the following subsection we give a brief recap on the notions of maximal regularity and the above mentioned conditions on $\Phi_{\tau}$ and $\Psi_{\tau}$, which can be conveniently phrased in terms of \emph{admissible} operators, a notion borrowed from infinite-dimensional systems theory \cite{Weiss89con}.

\subsection{Notions and basic properties}\label{sec_admiss}
In the following $\cS=(S(t))_{t\ge0}$ always refers to a strongly continuous semigroup on a Banach space $X$ with generator $A$. All normed spaces considered in this paper are assumed to be complex. For normed spaces $X$ and $Y$, $\mathcal{L}(X,Y)$ denotes the space of bounded linear operators from $X$ to $Y$, with the convention $\mathcal{L}(X)=\mathcal{L}(X,X)$. The domain and range of a linear, possibly unbounded, operator $B$ will be denoted by $D(B)$ and $\mathrm{ran}\,B$ respectively. Furthermore, let $\rho(B)$ refer to the resolvent set of $B$ and for $\lambda\in \rho(B)$ we write $R(\lambda,B)=(\lambda-B)^{-1}$ for the resolvent operator. We associate the following abstract Sobolev spaces with the semigroup $\cS$; the space $X_{1}=(D(A),\|\cdot\|_{D(A)})$, where $\|\cdot\|_{D(A)}$ refers to the graph norm of $A$, and $X_{-1}$, which is the completion of $X$ with respect to the norm $\|R(\lambda,A)\cdot\|$ for some fixed $\lambda\in \rho(A)$. It is well-known that $\cS$ can be uniquely extended to a strongly continuous semigroup $\cS_{-1}$ on $X_{-1}$ whose generator $A_{-1}$ extends $A$ and has domain $D(A_{-1})=X$. 
For an interval $I\subseteq \RR_+:=[0,\infty)$, $p\in [1,\infty]$ and some Banach space $U$, let $\rL^p(I,U)$, $\mathrm{Reg}(I,U)$ and $\mathrm{C}(I,U)$ refer to the spaces of Lebesgue-Bochner $p$-integrable  (equivalence classes of) functions, the regulated functions and the continuous functions, $f:I\to U$, respectively, with the usual convention for $p=\infty$. We equip both $\rC(I,U)$ and $\rReg(I,U)$ with the supremum norm. Sometimes we will use the place holder $\rF$ in $\rF(I,U)$ to formulate a statement for either of the above spaces.

\begin{definition} Let $\rF$ be either $\rL^{p}$, $p\in [1,\infty]$, or $\rC$ or $\rReg$.
	A strongly continuous semigroup $\cS:=(S(t))_{t\geq0}$ (or its generator $A$) is said to satisfy the \emph{maximal regularity property with respect to  $\rF$} or \emph{$\rF$-{maximal regularity}}, if for some $\tau>0$ and all $f\in \rF([0,\tau],X)$,
	it holds that  $(\cS\ast f)(t)\in D(A)$ for almost every $t\in (0,\tau)$ and  \[A(\cS\ast f)\in \rF([0,\tau],X),\] where 
	\begin{equation}
	(\cS\ast f)(t):=\int_{0}^t S(t-s)f(s) \dd s.
\end{equation}

\end{definition}
It is easy to see that $\cS$ has the $\mathrm{C}$-maximal-regularity property  if and only if $\cS\ast f\in \mathrm{C}([0,\tau],X_1)$, see also \cite{EG}. Furthermore, whenever $\cS$ has the $\mathrm{C}$-maximal-regularity for some $\tau>0$, then this holds for every $\tau>0$.

The following notions are central for this work.

\begin{definition}\label{def-control-admiss}
	Let $U,Y$ be Banach spaces and $\rF$ be either $\rL^{p}$, $\rReg$ or $\rC$. 
\begin{enumerate}
\item An operator $B\in \mathcal{L}(U,X_{-1})$ is called an \emph{$\rF$-admissible control operator} or \emph{$\rF$-admissible} for $\cS$, if for some (hence all) $\tau>0$ the mapping
	\begin{equation*}
	\Phi_{\tau}\colon \rF([0,\tau],U)\to X_{-1}, \quad
	u \mapsto\int_0^{\tau}S_{-1}(\tau-t)Bu(t)\dd t
	\end{equation*}
	has range in $X$, i.e.~$\ran(\Phi_{\tau})\subset X$. 
	
\item We call $C\in\mathcal{L}(X_{1},Y)$ an \emph{$\rF$-admissible observation operator} or \emph{$\rF$-admissible} for $\cS$, if for some (hence all) $\tau>0$ the operator
	\begin{equation*}
	\Psi_{\tau}\colon X_1\to \rF([0,\tau],Y),\quad x\mapsto CS(\cdot)x
	\end{equation*}
	has a bounded extension to $X$,  again denoted by $\Psi_{\tau}$. 
	\end{enumerate}
	If $\limsup_{\tau\to0^{+}}\|\Phi_{\tau}\|_{\cL(\rF([0,\tau],U),X)}=0$ or $\limsup_{\tau\to0^{+}}\|\Psi_{\tau}\|_{\cL(X,\rF([0,\tau],Y))}=0$, we say that $B$ or $C$ are \emph{zero-class $\rF$-admissible}, respectively.
\end{definition}

	The above notion of admissible operators, first coined by G.~Weiss \cite{Weiss89con,Weiss89obs}, plays an important role in the context of infinite-dimensional linear systems theory, and particularly, in the context of boundary control and observation, see also \cite{S}. 
	Note that by the Closed-Graph Theorem, $B\in \mathcal{L}(U,X_{-1})$ is $\rF$-admissible if and only if $\Phi_{\tau}$ is bounded from $ \rF([0,\tau],U)$ to $X$, i.e.\
	 there exists $K_{\tau}>0$ such that 
	\begin{equation*}
	\left\|\int_0^{\tau}S_{-1}(\tau-s)Bu(s)\dd s\right\|\leq K_{\tau} \|u\|_{\rF([0,\tau],U)}, \quad  u\in \rF([0,\tau],U).
	\end{equation*}
	This also shows that in the definition the norm $\|\Phi_{\tau}\|_{\cL(\rF([0,\tau],U),X)}$ in the above definition of zero-class admissibility is well-defined.
	Since for $p\in (1,\infty)$,
	\[\mathrm{C}(I,U)\subset \mathrm{Reg}(I,U)\subset \rL^{\infty}(I,U)\subset \rL^{p}(I,U)\subset \rL^{1}(I,U)\] for bounded intervals $I\subset\RR_+$, with continuous embeddings, there is a natural chain of implications for the property of admissible operators. 
	In particular, $B$ being a $\rC$-admissible control operator is the weakest property in the the scale of $\rF$-admissibility. Dually, any $\rF$-admissible observation operator  is $L^{1}$-admissible. 
	In the following we will only be interested in the cases where $B=A_{-1}$ is an admissible control operator or $C=A$ is an admissible observation operator.

\noindent
The following result, which is a slight extension of \cite[Proposition 16]{JacSchwZw}, marks the point of departure for Section \ref{Sec:MaxReg}.
The proof follows the same lines as in the cited reference and is therefore omitted.
\begin{proposition}\label{Aboundedregu}
	Let $A$ be the generator of a strongly continuous semigroup $\cS$. If either 
	\begin{itemize}
	\item $A$ is a zero-class $\rL^{1}$-admissible observation operator, or
	\item $A_{-1}$ is a zero-class $\rC$-admissible control operator, 
	\end{itemize}
	then $A$ is bounded.
\end{proposition}
In Example \ref{ex:Kato} below we will show that the assumption of zero-class admissibility cannot be dropped in the above proposition (in either case). We conclude this preparatory section with a result that gives an indication how the seemingly strong condition of $A_{-1}$ being an admissible control operator  relates to admissibility of general control operators.
\noindent

\begin{proposition}\label{prop:Adm}
Let $\cS$ be a strongly continuous semigroup on a Banach space $X$ with generator $A$. Let $\rF$ be a placeholder for either $\rC$, $\mathrm{Reg}$ or $\rL^{\infty}$. The following assertions are equivalent.
\begin{enumerate}
\item For every Banach space $U$, every operator $B\in \cL(U,X_{-1})$ is $\rF$-admissible.
\item $A_{-1}$ is $\rF$-admissible.
\end{enumerate}
\end{proposition}
\begin{proof}
Without loss of generality assume that $0\in \rho(A)$ and fix $t>0$. Note that the mapping $f\mapsto A_{-1}f(\cdot)$ is an isomorphism from $\rL^{\infty}((0,t),X)$ to $\rL^{\infty}((0,t),X_{-1})$. Assume that $A_{-1}$ is $\rL^{\infty}$-admissible and consider $B\in\cL(U,X_{-1})$ for some Banach space $U$. Since for any $u\in \rL^{\infty}((0,t),U)$ it holds that $Bu(s)=A_{-1}\tilde{u}(s)$ where $\tilde{u}=A_{-1}^{-1}Bu\in \rL^{\infty}((0,t),X)$, we conclude that $B$ is $\rL^{\infty}$-admissible. The converse is clear since $A_{-1}\in \cL(X,X_{-1})$. 

\end{proof}
Note that the variant of Proposition \ref{prop:Adm} for $\rF=\rL^{p}$, $p<\infty$, is trivial, since, by H\"older's inequality, Proposition \ref{Aboundedregu} implies that $A_{-1}$ is $\rL^{p}$-admissible if and only if $A$ is bounded.

\section{Maximal regularity and admissible generators}
\label{Sec:MaxReg}
{

The following two propositions show that  maximal regularity  and admissibility with respect to continuous, regulated functions and $\rL^{1}$-functions are closely related.

\begin{proposition}\label{propMR1a}
 Let $\cS$ be a strongly continuous semigroup on a Banach space $X$ with generator $A$. Then the following assertions hold:
 \begin{enumerate}
 \item\label{propMR1} $\cS$ satisfies  ${\rL}^{1}$-maximal-regularity if and only if $A$ is $\rL^{1}$-admissible. 
 \item\label{propMR2} $A_{-1}$ is $\rL^{p}$-admissible for some $p\in(1,\infty)$ if and only if $A\in\mathcal{L}(X)$.
\end{enumerate}
\end{proposition}
\begin{proof}
The first assertion was proved in \cite[Theorem 3.6]{KaltPort08} (without explicitly using the notion of admissible operators) by rescaling the semigroup.
Assertion \eqref{propMR2} is immediate from Proposition \ref{Aboundedregu} and H\"older's inequality.
\end{proof}

Note that by Guerre-Delabri{\`e}re's result it holds that if $\cS$ satisfies  ${\rL}^{1}$-maximal-regularity and $A$ is unbounded, then $X$ must contain a complemented copy of $\ell^{1}$, see \cite{GuerDela95} and \cite{KaltPort08}. 

\begin{proposition}\label{propMR}
	Let $\cS$ be a strongly continuous semigroup on a Banach space $X$ with generator $A$. Then the following assertions are equivalent:
	\begin{enumerate}
		\item \label{propMR11}$\cS$ satisfies ${\rC}$-maximal-regularity,
		\item \label{propMR12} $A_{-1}$ is $\rC$-admissible,
		\item \label{propMR13} $\cS$ is of bounded semivariation, i.e.\ \eqref{eq:SV} holds for all $\tau>0$,
		\item  \label{propMR14} $\cS$ satisfies ${\rReg}$-maximal-regularity,
		\item  \label{propMR15} $A_{-1}$ is $\rReg$-admissible.	
	\end{enumerate}
\end{proposition}
\begin{proof}
The equivalences \eqref{propMR11}$\Leftrightarrow$\eqref{propMR2}$\Leftrightarrow$\eqref{propMR13} are shown in \cite[Lemma 3.1 and Proposition 3.1]{Travis}.  The implication \eqref{propMR14}$\Rightarrow$\eqref{propMR15} is easy to see from the definitions and \eqref{propMR15}$\Rightarrow$\eqref{propMR14} is a consequence of \cite[Theorem 4.3.1]{S}, which even implies that $Ax$ is continuous for any $f\in\rReg([0,\tau],X)$, where $x$ is the mild solution to equation \eqref{acp}. Since \eqref{propMR15} trivially implies \eqref{propMR12}, it follows that \eqref{propMR15}$\Rightarrow$\eqref{propMR13}.\newline It remains to show that \eqref{propMR13}$\Rightarrow$\eqref{propMR15}.
Let $f\colon [0,\tau]\to X$, $\tau>0$, be an arbitrary regulated function and suppose that $\cS$ is of bounded semivariation on $[0,\tau]$.
Then there exists a sequence of step functions $(f_n)_{n\in\NN}$ represented by
\[ f_n(s):=\sum_{i=1}^{n}f(d_i^n)\chi_{ (d_{i-1}^n,d_i^n)}(s),\qquad s\in [0,\tau], \]
which converges uniformly to $f$ and where $0=d_{0}^n<d_{1}^n<\dots<d_{n}^n=\tau$. Define $g_n(s)=S(\tau-s)f(d_i^n)$ for $d_{i-1}^n<s\leq d_i^n$, $i=1,\dots,n$, and $g_n(0)=S(\tau)f(0)$. Because $f$ is bounded and $\cS$ is strongly continuous, we have that $(g_n)_{n\in\NN}$ is uniformly bounded and
 converges uniformly to $s\mapsto S(\tau-s)f(s)$ for $n\to\infty$. Therefore, 
\[ \lim\limits_{n\to\infty}\int_0^\tau g_n(s)\dd s=\int_0^\tau S(\tau-s)f(s)\dd s. \]
Because $\int_0^\tau g_n(s)\dd s\in D(A)$ we can calculate,
\begin{align*}
A\int_0^\tau g_n(s)\dd s={}&A\sum_{i=1}^{n}\int_{d_{i-1}^n}^{d_i^n}g_n(s)\dd s
\\={}&\sum_{i=1}^{n} \left[ S(\tau-d_{i}^n)-S(\tau-d_{i-1}^n) \right]f_{n}(d_i^n)\\=:{}&h_n.
\end{align*}
Since $\cS$ is of bounded semivariation, we have for $n,m\in\NN$ that
\[ \|h_n-h_m\|_X\leq \operatorname{var}_0^\tau(\cS) \|f_n-f_m\|_{\infty},\]
where  $\operatorname{var}_0^\tau(\cS)$ is defined as in \eqref{eq:SV}.  
Since $(f_n)_{n\in\NN}$ is a Cauchy sequence, it follows that $(h_n)_{n\in\NN}$ is a Cauchy sequence of $X$ and thus converges to a limit in $X$. 
Hence, as $A$ is closed,
\[ \int_0^\tau S(\tau-s)f(s)\dd s\in D(A), \]
and therefore,
\[ \int_0^\tau S_{-1}(\tau-s)A_{-1}f(s)\dd s=A\int_0^\tau S(\tau-s)f(s)\dd s\in X. \qedhere\]
\end{proof}

 The following example has been used in the context of $\rC$-maximal regularity several times, \cite{Bail80,EG}, and seems to go back T.~Kato\footnote{see A.~Pazy's review of \cite{Bail80}, MR0577152, on MathSciNet}. We use it to show that an analogous statement as Proposition \ref{propMR} for $\rL^{\infty}$ does not hold. 
\begin{example}\label{ex:Kato}
	Let $X=c_0(\NN)$ and $Ax=\sum_{n=1}^{\infty}-nx_ne_n$ with $D(A)=\{x\in X \, :\, \sum_{n=1}^{\infty}-nx_ne_n\in c_0(\NN)\}$ and where $(e_{n})_{n\in\mathbb{N}}$ refers to the canonical basis. 
	It is easy to see that $A$ generates an exponentially stable strongly continuous semigroup $\cS=(S(t))_{t\geq 0}$ given by 
	\[S(t)x=\sum_{n=1}^{\infty}e^{-nt}x_ne_n \]
	(see e.g. \cite[Example 4.7 iii), Chapter I]{EN}). 
	Let now $B=-A_{-1}\in \cL(X,X_{-1})$.
	Define $u\in \rL^{\infty}([0,\tau],c_0(\NN))$ by $u(s)=\sum_{n=1}^{\infty}(u(s))_ne_n$, where 
	\[
	(u(s))_n:=\begin{cases} 1 & \mbox{if } s\in[\tau-\frac{1}{n},\tau-\frac{1}{2n}] \mbox{ and }\frac{1}{n}<\tau, \\
	0 & \mbox{ otherwise.}\end{cases}
	\]
	The element $f=\int_0^{\tau}S_{-1}(\tau-s)Bu(s)\dd s=\int_0^{\tau}S_{-1}(s)Bu(\tau-s)\dd s$ in $X_{-1}$ can be represented by a sequence $(f_n)_n$ and we can calculate for all $n\in\NN$ with $\frac{1}{n}<\tau$,
	\[
	f_n=\int_0^{\tau}e^{-ns}n(u(\tau-s))_n\dd s=\int_{\frac{1}{2n}}^{\frac 1n}ne^{-ns}\dd s
	=-(e^{-1}-e^{-\frac 12})>\frac 15.
	\]
	This shows $f\notin c_0(\NN)$ and therefore $B$ is not $\rL^{\infty}$-admissible.
	On the other hand, it is easy to see that $\cS$ satisfies $\rF$-maximal-regularity for $\rF=\rC$, $\rReg$ and $\rL^\infty$, see e.g.\ \cite{EG}, and thus $B=-A_{-1}$ is $\rReg$-admissible  and hence $\rC$-admissible by Proposition \ref{propMR}.
	Since $A$ is obviously unbounded, Proposition \ref{Aboundedregu} shows that $B$ is not zero-class $\rF$-admissible with $\rF$ equal to $\rReg$ or $\rC$.
		\end{example}

	\rem{
		\medskip
		
		Now define $\tilde{B}x=\sum_{n=1}^{\infty}\frac{n}{\log n}x_ne_n$ which is again a positive operator in $\cL(U,X_{-1})$. Let $\tilde{f}=\int_0^{\tau}S_{-1}(\tau-s)\tilde{B}u(s)\dd s=\int_0^{\tau}S_{-1}(s)\tilde{B}u(\tau-s)\dd s$ and we calculate
		\[
		|\tilde{f}_ne_n|=\left|\int_0^{\tau}e^{-ns}n(u(\tau-s))_n\dd s\right|\leq \|u\|_{\infty}\left|\int_{0}^{1}\frac{n}{\log n}e^{-ns}\dd s\right| \]
		\[ =-\frac{1}{\log n}(e^{-\tau}-1)=\frac{1}{\log n}(1-e^{-\tau}).
		\]
		Therefore $f\in c_0(\NN)$ and we see that $\tilde{B}$ is a zero-class $\rL^{\infty}$-admissible operator, but not $\rL^p$-admissible for all $p\in[0,\infty[$ since
		\[
		\|\tilde{f}\|^p_l
		\]}

\begin{remark}
It seems to have been unnoticed in the literature that the simple Example \ref{ex:Kato} answers an old question posed by G.~Weiss in \cite[Remark 6.10]{Weiss89obs} about whether  $\rL^{1}$-admissibility of $B^*$ with respect to the dual semigroup $\cS^*$ implies that $B$ is $\rL^{\infty}$-admissible for $\cS$ for general $B\in \mathcal{L}(U,X_{-1})$ and where $B^{*}$ is to be understood as an admissible observation operator. In the setting of the example, the dual semigroup on $X^{*}=\ell^{1}(\mathbb{N})$ is given by $S^{*}(t)x=\sum_{n}\mathrm{e}^{-nt}x_{n}e_{n}$ for any $x=\sum_{n}x_{n}e_{n}$ and where $(e_{n})_{n\in\mathbb{N}}$ (again) refers to the canonical basis of $\ell^{1}$. It is easy to check that $A^{*}=(A_{-1})^{*}$ is an $\rL^{1}$-admissible observation operator for $\cS^{*}$. On the other hand, $B=A_{-1}$ is not $\rL^{\infty}$-admissible for $\cS$ as shown above. 
Furthermore, the example shows that $\rReg$-admissibility does not imply zero-class $\rReg$ admissibility. Therefore, the assumption of ``zero class admissibility'' in  Proposition \ref{Aboundedregu} cannot be relaxed to plain ``admissibility''. On the other hand, the existence of a $\rReg$-admissible operator $B$ which is not zero-class $\rReg$-admissible also establishes a counterexample in the context of \emph{input-to-state stability} for infinite-dimensional linear systems, see e.g.\ \cite{MiroPrie19,JacNabPartSchw}: By \cite[Proposition 2.10]{JacNabPartSchw}, it shows that there exists a system $(A,B)$ which is input-to-state stable, but not integral input-to-state stable both with respect to respect to $\rReg$-input functions.
\end{remark}
In order to proceed to our main result, we need to discuss Baillon's theorem on maximal regularity with respect to $\rC$ in more detail. 
The following proposition was derived within the proof of Baillon's theorem \cite{Bail80}, see also \cite{EG}. Since we need it formulated explicitly, we sketch a short argument based on a classical characterization of spaces containing $c_{0}$ due to Bessaga--Pe\l cy\'nski \cite{BessPelc58}, see also \cite{EG}.
\begin{proposition}[Baillon's theorem and Baillon spaces]\label{thm:Baillon}
Let $A$ generate  a strongly continuous semigroup $\cS$ on a Banach space $X$. If $\cS$ satisfies  {\rm C}-maximal-regularity and $A$ is unbounded, then $X$ contains an isomorphic copy of $c_{0}$. More precisely, there exists a  sequence $(z_{n})_{n\in\mathbb{N}}$ in $X$ with the following properties:
\begin{enumerate}
\item $Z:=\overline{\mathrm{span}}(z_{n})_{n\in\mathbb{N}}$ is isomorphic to $c_{0}$, \label{it22}
\item $(z_{n})_{n\in\mathbb{N}}$ is a Schauder basis of $Z$, \label{it21}
\item $0<\inf_{n}\|z_{n}\|\leq\sup_{n}\|z_{n}\|<\infty$, \label{it23}
\item $\lim_{n\to\infty}R(\lambda,A)z_{n}=0$ for any $\lambda\in \rho(A)$. \label{it24}
\end{enumerate}
We call such a space $Z$ a \emph{Baillon space}.
\end{proposition}
\begin{proof}
To show the existence of a sequence $(z_{n})_{n\in\mathbb{N}}$ satisfying (\ref{it21})--(\ref{it23}), it suffices to find a sequence $(x_{n})_{n\in\mathbb{N}}$ and a constant $M>0$ such that \[ \inf_{n}\|x_{n}\|>0 \quad\text{ and } \quad \left\|\sum_{j=0}^m x_{n_{j}}\right\|\le M, \qquad m\in \mathbb{N},\]
for any increasing sequence $(n_{j})_{j}$ of positive integers, see \cite[Theorem 0.1]{EG}. Following \cite{Bail80} or \cite{EG}---using that $A$ is unbounded and  that $\cS$ satisfies $\rC$-maximal-regularity---a possible choice is given by $x_{n}=t_{n}AS(t_{n})y_{n}$ with suitably chosen sequences $(t_{n})_{n\in\mathbb{N}}$ \fs{in $(0,1)$ and $(y_{n})_{n\in\mathbb{N}}$ in $X$ with $\sum_{n=1}^{\infty}t_{n}<\infty$} and $\|y_{n}\|=1$, $n\in\mathbb{N}$. By \cite[Cor.~1 and Lemma 3]{BessPelc58}, see also \cite[Lemma D.2 and Theorem D.3]{AreBatHieNeu}, the sequence $(z_{n})_{n\in\mathbb{N}}$ can then be derived as a block basis from $(x_{n})_{n\in\mathbb{N}}$, i.e.\ there exists an increasing sequence of positive integers $(p_{k})_{k\in\mathbb{N}}$ and a sequence of positive numbers $(s_{k})_{k\in\mathbb{N}}$ such that
\[z_{k}=\sum_{n=p_{k}+1}^{p_{k+1}}s_{n}x_{n},\qquad k\in \mathbb{N},\]
satisfies (\ref{it21})--(\ref{it23}). By the proof of Lemma D.2 in \cite{AreBatHieNeu}, the sequence $(s_{k})_{k\in\mathbb{N}}$ can be chosen to be bounded. 
To see Assertion \eqref{it24}, fix $\lambda\in\rho(A)$ and note that
\[\fs{\|R(\lambda,A)z_{n}\| \leq \sum_{n=p_{k}+1}^{p_{k+1}}t_{n}\|(\lambda R(\lambda,A)-I)S(t_{n})y_{n}\|\longrightarrow0 \quad (n\to\infty).}\qedhere\]  
\end{proof}

In particular, Baillon's result shows that if $A_{-1}$ is $\rL^{\infty}$-admissible and $A$ is unbounded, then $X$ must contain $c_{0}$ and therefore, e.g., fails to be reflexive. This, however, does not exploit the difference between $\rL^{\infty}$-admissibility and $\rC$-admissibility. A step towards this is achieved in the following result.
\begin{theorem}\label{thm:Fav}
Let $\cS$ be a strongly continuous semigroup on a Banach space $X$. If $A_{-1}$ is $\rL^{\infty}$-admissible then 
\[Fav(\cS)=D(A),\] where $Fav(\cS):=\{x\in X\colon\limsup_{t\to0^{+}}\frac{1}{t}\|S(t)x-x\|<\infty\}$ refers to the \emph{Favard space} of $\cS$.
\end{theorem}
\begin{proof} Without loss of generality we may assume that $0\in\rho(A)$. 
Let $x\in Fav(\cS)$.
It is well-known, e.g.\ \cite[Theorem 3.2.8]{Nee}, that there exists a bounded sequence
$(y_{n})_{n\in\mathbb{N}}$ in $X$ such that $\lim_{n\to\infty}A^{-1}y_{n}=x$. Let $(t_{n})_{\in\mathbb{N}}$ be a strictly decreasing sequence of positive numbers with $t_{0}=1$ and $\lim_{n\to\infty}t_{n}=0$. Define $u:[0,1)\to X$ by
\[u(s)=y_{n}\qquad\text{for } s\in [1-t_{n},1-t_{n+1}), n\in\mathbb{N}_{0}.\] 
Clearly, $u\in {\rL}^{\infty}((0,1),X)$ since $(y_{n})$ is bounded. We have by assumption that
\begin{align*}
 \int_{0}^{1}S_{-1}(1-s)A_{-1}u(s)\mathrm{d}s={}&\int_{0}^{1}A_{-1}S_{-1}(s)u(1-s)\mathrm{d}s\\
 	={}&\fs{\sum_{n=0}^{\infty}\int_{t_{n+1}}^{t_{n}}A_{-1}S_{-1}(s)u(1-s)\mathrm{d}s}\\
 	={}&\sum_{n=0}^{\infty}\int_{t_{n+1}}^{t_{n}}\frac{\mathrm{d}}{\mathrm{d}s}S_{-1}(s)u(1-s)\mathrm{d}s\\
	={}&\sum_{n=0}^{\infty}(S(t_{n})y_{n}-S(t_{n+1})y_{n}),
 \end{align*}
 where the involved integrals and sums converge in $X_{-1}$ with limit in $X$. Upon considering a subsequence, assume that $\sum_{n=0}^{\infty}t_{n}^{-1}\|y_{n+1}-y_{n}\|_{-1}<\infty$. Without loss of generality we can set $y_{0}=0$. By using that the semigroup is in fact analytic, which follows from the assumption that $A_{-1}$ is $\rL^{\infty}$-admissible, e.g.\ by \cite[Proposition 9]{JacSchwZw} and Proposition \ref{prop:Adm},
 \begin{align*}
 \|S(t_{n+1})(y_{n}-y_{n+1})\|\leq{}& \|AS(t_{n+1})\|\|y_{n}-y_{n+1}\|_{-1}\\
 \leq{}& Ct_{n+1}^{-1}\|y_{n}-y_{n+1}\|_{-1}, 
 \end{align*}
 and thus $\sum_{n=0}^{\infty}S(t_{n+1})(y_{n}-y_{n+1})$ converges in $X$ absolutely. Combining this with the above shows that
 \begin{align*}
 S(t_{N})y_{N}=&{}\sum_{n=0}^{N-1}S(t_{n+1})y_{n+1}-S(t_{n})y_{n}\\
  ={}&\fs{\sum_{n=0}^{N-1}(S(t_{n+1})y_{n}-S(t_{n})y_{n}) -\sum_{n=0}^{N-1}S(t_{n+1})(y_{n}-y_{n+1})}
 \end{align*}
 converges for $N\to\infty$ in the $X_{-1}$-norm with a limit in $X$. Since
 \begin{align*}
 \|S(t_{N})y_{N}-A_{-1}x\|_{-1}\leq{}&\|S(t_{N})y_{N}-S(t_{N})A_{-1}x\|_{-1}+\|(S(t_{N})-I)A_{-1}x\|_{-1} \\
 \leq{}&M\|y_{N}-A_{-1}x\|_{-1}+\|S(t_{N})A_{-1}x-A_{-1}x\|_{-1},
 \end{align*}
we have that the $X_{-1}$-limit of $S(t_{N})y_{N}$  equals $A_{-1}x$ which is the $X_{-1}$-limit of the sequence $(y_{n})$.
 Thus, $A_{-1}x\in X$, and therefore $x=A^{-1}A_{-1}x\in D(A)$.
Hence, $Fav(\cS)=D(A)$ since the other inclusion holds trivially. 
 \end{proof}
Note that the assumption of $\rL^{\infty}$-admissibility in Theorem \ref{thm:Fav} cannot be relaxed to $C$-admissibility; see Example \ref{ex:Kato} where $Fav(\cS)\neq D(A)$, which can be checked directly, or by Theorem \ref{thm:main} below.

\fs{For what follows it will be crucial to use the sun-dual space \[X^{\odot}=\{x^{*}\in X^{*}\colon \lim_{t\to0^{+}}\|S^{*}(t)x^{*}-x^{*}\|=0\}\] associated with the semigroup $\cS$ and which is a closed subspace of $X^{*}$. This allows us to define the following norm,
\begin{equation}
\label{eq:normX}
 |||x|||=\sup_{x^{\odot}\in X^{\odot}, \|x^{\odot}\|_{X^{*}}\leq 1}|\langle x,x^{\odot}\rangle|, \quad x\in X,
\end{equation}
which is known to be equivalent to $\|\cdot\|$ on $X$, see e.g.\ \cite[p.~7]{Nee}.} Furthermore, it is clear that the mapping
\begin{equation}\label{eq:j}
j:X\to X^{\odot*}, x\mapsto (x^{\odot}\mapsto \langle x^{\odot},x\rangle)
\end{equation}
is an isometry when $X$ is equipped with $|||\cdot|||$. Note that $j$ is in general not isometric when the norm $\|\cdot\|$ is considered on $X$. However, if $X^{\odot}=X^{*}$, then $j$ equals the canonical isometry from $X$ in its bidual.

By a result due to van Neerven, \cite[Theorem 3.2]{vNee91a}, see also \cite[Theorem 3.2.9]{Nee}, the property that $Fav(\cS)=D(A)$ is equivalent to the condition that the set $R(\lambda,A)K_{(X,|||\cdot|||)}$ is closed in $X$ for some (hence all) $\lambda\in \rho(A)$, where $K_{(X,|||\cdot|||)}$ refers to $\{x\in X\colon |||x|||\leq 1\}$. We will  employ this fact in the following. We emphasize that the use of the $|||\cdot|||$-norm is crucial here as the corresponding statement involving the $\|\cdot\|$-norm does not hold in general, see \cite[Example 3.2.11]{Nee}.
\\ It is not hard to see that $Fav(\cS)=D(A)$ is satisfied for all semigroups whenever $X$ is reflexive, see e.g.\ \cite{Nee} and \cite[Corollary II.5.21]{EN}. This also shows that the converse of Theorem \ref{thm:Fav} is not true. Furthermore, the case $X=c_{0}$ is special as $Fav(\cS)=D(A)$ implies that $\cS$ is uniformly continuous then, \cite[Theorem 3.2.10]{Nee}.
The latter result rests on non-trivial fact of the geometry on $c_{0}$, which, loosely speaking, guarantees that a given operator $R:c_{0}\to Y$, with some arbitrary Banach space $Y$, is an isomorphism under comparably little information on $R$. To make this more explicit, let us introduce the following notions, which will be used subsequently. 
\begin{definition}[Semi-embeddings and $G_{\delta}$-embeddings]
Let $X$ and $Y$ be Banach spaces. An injective bounded linear operator $R:X\to Y$ is called 
\begin{itemize}
\item  a \emph{semi-embedding} if $R(\{x\in X\colon \|x\|\leq 1\})$ is closed in $Y$; or
\item  a \emph{$G_{\delta}$-embedding} if $R(M)$ is a $G_{\delta}$-set for any closed bounded set $M$ in $X$.
\end{itemize}
\end{definition}
Semi-embeddings were first studied by Lotz, Peck and Porta in \cite{LotzPeckPort79} and further investigated by Bourgain and Rosenthal in \cite{BourRose83}, who introduced the notion of a $G_{\delta}$-embedding. The latter was partially motivated by the fact that the property of $R$ being a semi-embedding is neither inherited by restrictions to closed subspaces nor invariant under isomorphisms. We collect the following facts for later reference.
\begin{lemma}[Bourgain--Rosenthal \cite{BourRose83}]\label{lem:BR}
Let $X$, $Y$ be Banach spaces and $R\in \mathcal{L}(X,Y)$. Then the following assertions hold.
\begin{enumerate}
\item\label{lem:BR1} If $R$ is a semi-embedding and $X$ is separable then  $R$ is a $G_{\delta}$-embedding. 
\item If $R$ is a $G_{\delta}$-embedding, then $R|_{Z}:Z\to Y$ and $RS$ are $G_{\delta}$-embeddings, for any closed subspace $Z\subset X$ and any isomorphism $S:W\to X$.  \label{lem:BR2}
\item  \label{lem:BR3} If $X=c_{0}$ and $R$ is a $G_{\delta}$-embedding, then $R$ is bounded from below, i.e.\ there exists $C>0$ such that 
\[ \|Rx\|\geq C\|x\|, \qquad x\in X.\]
\end{enumerate}
\end{lemma}
The proofs of \eqref{lem:BR1} and \eqref{lem:BR3} can be found in \cite[Prop.~1.8 and Prop.~2.2]{BourRose83}. The other assertion is clear by definition.

We are now able to prove our main result.

\begin{theorem}\label{thm:main}
Let $\cS$ be a strongly continuous semigroup on a Banach space $X$ with generator $A$.
Then the following assertions are equivalent 
\begin{enumerate}
\item\label{main:it1} $A_{-1}$ is $\rL^{\infty}$-admissible,
\item \label{main:it2}  $Fav(\cS)=D(A)$ and $\cS$ satisfies $\rC$-maximal-regularity,
\item\label{main:it3}  $A$ is a bounded operator.
\end{enumerate}
\end{theorem}
\begin{proof}
The  implication \eqref{main:it3}$\Rightarrow$\eqref{main:it1} is clear as admissibility of $A_{-1}=A$ is trivial when $A$ is bounded. Furthermore,  \eqref{main:it1}$\Rightarrow$\eqref{main:it2}  follows by Theorem \ref{thm:Fav} and since $\rL^{\infty}$-admissibility implies $\rC$-admissibility of $A_{-1}$ which is equivalent to $\cS$ satisfying $\rC$-maximal-regularity, Proposition \ref{propMR}.
\\
Hence, it remains to show \eqref{main:it2}$\Rightarrow$\eqref{main:it3}.
\fs{We assume that} $Fav(\cS)=D(A)$ and that $\cS$ satisfies $\rC$-maximal-regularity. Suppose that $A$ is unbounded. Thus we may  consider the 
 Baillon space $Z$ as given in Proposition \ref{thm:Baillon}. Let $\tilde{Z}$ be the smallest closed $\cS$-invariant subspace containing $Z$.  Since 
\[\tilde{Z}=\overline{\{S(t)z\colon z\in Z,t\ge0\}},\]
it is easy to see that $\tilde{Z}$ is separable, since $Z$ is separable. \fs{Let $\tilde{\cS}$ denote the restricted semigroup on $\tilde{Z}$ whose generator $\tilde{A}$ is given by $D(\tilde{A})=D(A)\cap \tilde{Z}$ and $\tilde{A}z=Az$ for $z\in D(\tilde{A})$.  Since by assumption $D(A)=Fav(\cS)$, it holds that $D(\tilde{A})=Fav(\tilde{\cS})$ 
and thus, by \cite[Theorem 3.2.9]{Nee}, $R(\lambda,\tilde{A})K=R(\lambda,A)K$ is $\|\cdot\|$-norm-closed in $X$, where $K=\{x\in \tilde{Z}\colon |||x|||\leq 1\}$.
 Therefore, $R(\lambda,A)|_{\tilde{Z}}$ is a semi-embedding from the separable space $(\tilde{Z},|||\cdot|||)$ to $X$ and, hence, 
 a $G_{\delta}$-embedding by Lemma \ref{lem:BR}\eqref{lem:BR1}. }Since the property of being a $G_{\delta}$-embedding is inherited by restrictions on closed subspaces and invariant under isomorphisms, Lemma \ref{lem:BR}\eqref{lem:BR2}, we infer that also $R(\lambda,A)|_{Z}:Z\to X$ is a $G_{\delta}$-embedding, where $Z$ is now equipped with the norm $\|\cdot\|$, which is equivalent to $|||\cdot|||$.
Since $Z$ is isomorphic to $c_{0}$, we conclude by Bourgain--Rosenthal, Lemma \ref{lem:BR}\eqref{lem:BR3} and \ref{lem:BR}\eqref{lem:BR2}, that $R(\lambda,A)|_{Z}$ is bounded from below. Hence, there exists a constant $C>0$ such that 
\[ \|R(\lambda,A)x\|\geq C\|x\|, \qquad x\in Z.\]
 This, however, contradicts the property of a Baillon space that $R(\lambda,A)z_{n}$ tends to $0$ as $n\to\infty$, Proposition \ref{thm:Baillon}\eqref{it24}, where $(z_{n})_{n\in\mathbb{N}}$ is the sequence spanning $Z$, for which it holds that $\inf_{n}\|z_{n}\|>0$, Proposition  \ref{thm:Baillon}\eqref{it23}.
\end{proof}
In the study of adjoint semigroups (on non-reflexive spaces), the notion of sun-reflexivity---linking the space and the semigroup---is classical. Recall that $X$ is sun-reflexive (or ``$\odot$-reflexive'') for a given strongly continuous semigroup $\cS$ on $X$ if the isometry $j$ defined in \eqref{eq:j} maps $X$ onto $X^{\odot\odot}=(X^{\odot})^{\odot}$. Obviously, if $X$ is reflexive, then $X$ is sun-reflexive with respect to any strongly continuous semigroup $\cS$ on $X$. Also note that by de Pagter's result \cite{dePa89}, sun-reflexivity can be reformulated by the condition that the resolvent is weakly compact.
In the following we show that if $X$ is non-reflexive, but sun-reflexive with respect to $\cS$, then $A_{-1}$ cannot be $\rL^{\infty}$-admissible.
\begin{corollary}\label{cor:MaxReg}
Let $\cS$ be a strongly continuous semigroup such that $X$ is sun-reflexive. If $A_{-1}$ is $\rL^{\infty}$-admissible, then $A$ is bounded and $X$ is reflexive.
\end{corollary}
\begin{proof}
By Theorem \ref{thm:main}, we only have to show that $X$ is reflexive. This, however, is clear since $\cS$ is uniformly continuous which implies that $X^{\odot\odot}=X^{**}$ and that $j$ defined in \eqref{eq:j} equals the canonical embedding of $X$ in $X^{**}$.
\end{proof}
We point out that Corollary \ref{cor:MaxReg} can also be proved without referring to Theorem \ref{thm:main}. Indeed, by Theorem \ref{thm:Fav}, we know that $Fav(\cS)=D(A)$. Since $X$ is sun-reflexive,  this fact together with \cite[Theorem 6.2.14]{Nee} implies that $X$ must have the Radon-Nikodym property. Now $A$ must be bounded because otherwise Baillon's theorem shows that $X$ contains $c_{0}$ which contradicts that $X$ has the Radon-Nikodym property. That  $X$ is reflexive now follows in the same way as in  the other proof of the corollary.

\section{Comments on the discrete-time case}\label{sec:discrete}
The study of discrete-time versions of maximal regularity was initiated by Blunck \cite{Blun01b,Blun01a} and, with a focus on the extremal cases $\ell^{\infty}$ and $\ell^{1}$, further studied by Kalton and Portal in \cite{Port03} and \cite{KaltPort08}. For a power-bounded operator $T\in\mathcal{L}(X)$, we consider the solutions to 
\begin{align}\label{eq:discrete}
	x_{n}={}&Tx_{n-1}+Bu_{n},\qquad n\in\NN,\\
	x_{0}={}&0,\notag
\end{align}
where $(u_{n})_{n\in\NN}$ is in $\ell^{\infty}(\NN,X)$, the space of $X$-valued bounded sequences. 

\begin{definition}[discrete-time admissibility]
Let $T\in \mathcal{L}(X)$ be power-bounded, and $p\in[1,\infty]$. We say that $B\in \mathcal{L}(U,X)$ is an \emph{$\ell^{p}$-admissible control operator} for $T$ if  there exists $\kappa>0$ such that
\[ \left\|\sum_{k=1}^{n}T^{n-k}Bu_{k}\right\|\leq \kappa\|u\|_{\ell^{p}(\{1,\dots,n\},U)},\quad \forall n\in\NN, u\in \ell^{p}(\NN,U).\]
Furthermore, $C\in \mathcal{L}(X,Y)$ is called an \emph{$\ell^{p}$-admissible observation operator} for $T$ if  there exists $\tilde{\kappa}>0$ such that
\[ \left\|\left(CT^{k}x\right)_{k\in\NN}\right\|_{\ell^{p}(\NN,X)} \leq \tilde{\kappa}\|x\|,\qquad x\in X.\]
\end{definition}
We will always explicitly state whether an operator is $\ell^{p}$-admissible as observation or control operator in order to avoid confusion. This is contrast to the notation Section \ref{Sec:MaxReg}, where this was clear from the context. 
\begin{remark}\begin{itemize}
\item 
Note that by a uniform boundedness principle argument, if $p\in [1,\infty)$, an operator $B\in \mathcal{L}(U,X)$ is an $\ell^{p}$-admissible control operator if and only if the limit $\lim_{n\to\infty}\sum_{k=1}^{n}T^{n-k}Bu_{k}$ exists and
\[ \left\|\lim_{n\to\infty}\sum_{k=1}^{n}T^{n-k}Bu_{k}\right\| \leq \kappa \|u\|_{\ell^{p}},\quad u\in\ell^{p}(\NN,X).\]
For $p=\infty$, an analogous statement holds with $\ell^{\infty}$ replaced by $c_{0}$. 
\item Comparing the definitions of discrete-time and continuous-time admissibility leads to the following observation. Whereas our definition for continuous-time deals with fixed (and finite) times $\tau>0$, the property in the discrete case is connected to uniform estimates in $n$. 
We point out that, without loss of generality, we could have restricted ourselves to \emph{infinite-time admissible operators} in Section \ref{Sec:MaxReg} as well.
\end{itemize}
\end{remark}

On the other hand, the operator $T$ is said to satisfy \emph{$\ell^{p}$-maximal regularity} for $p\in[1,\infty]$ if for $B=I$ the solution $(x_{n})_{n\in\NN}$ to \eqref{eq:discrete} satisfies
\[(x_{n}-x_{n-1})_{n\in\NN}\in \ell^{p}(\mathbb{N},X),\qquad \forall (u_{n})_{n\in\NN}\in \ell^{p}(\mathbb{N},X).\]
Note that this is equivalent to the existence of some constant $\kappa>0$ such that 
\begin{equation}\label{eq:discrmaxreg}
\left\| \left(\sum_{k=1}^{n}T^{n-k}(I-T)u_{k}\right)_{n\in\NN}\right\|_{\ell^{p}(\NN,X)}\leq \kappa \|u\|_{\ell^{p}(\NN,X)},\quad u\in \ell^{p}(\NN,X).
\end{equation}
The following result was implicitly shown in \cite{KaltPort08} without using the notion of admissibility.
\begin{proposition}\label{prop:discr}
	Let $T\in\mathcal{L}(X)$ be power-bounded. Then the following assertions are equivalent.
	\begin{enumerate}
		\item\label{prop:discri1}	 $T$ satisfies $\ell^{\infty}$-maximal regularity,		
		\item\label{prop:discri2}      $T^{*}$ satisfies $\ell^{1}$-maximal regularity,
		\item\label{prop:discri3}      $I-T$ is an $\ell^{\infty}$-admissible control operator,
		\item\label{prop:discri4}      $I-T^{*}$ is an $\ell^{1}$-admissible observation operator.
	\end{enumerate}
\end{proposition}
\begin{proof}
	The equivalences \eqref{prop:discri1}$\Leftrightarrow$\eqref{prop:discri2}$\Leftrightarrow$\eqref{prop:discri4} follow from \cite[Proposition 2.3 and Theorem 3.1]{KaltPort08}. The proof of \eqref{prop:discri1}$\Leftrightarrow$\eqref{prop:discri3} is clear from the definitions.
	\end{proof}
	
	The duality in Proposition \ref{prop:discr} between maximal-regularity and admissibility with respect to $\ell^{1}$ and $\ell^{\infty}$ may come as a surprise when compared to the continuous-time situation, where such a result does not hold. The reason for this rests on the fact that there is no difference between $c_{0}$-maximal regularity (replacing $\ell^{\infty}$ by $c_{0}$ in \eqref{eq:discrmaxreg}) and $\ell^{\infty}$-maximal regularity. 

In \cite[Theorem 3.5]{KaltPort08} Kalton and Portal show a version of Baillon's result for discrete-time. More precisely, they prove that if $X$ does not contain $c_{0}$ and $T$ satisfies $\ell^{\infty}$-maximal regularity, then $X=X_{1}+X_{2}$ for $T$-invariant closed subspaces $X_{1}$ and $X_{2}$ such that $T|_{X_{1}}=I_{X_{1}}$ and $r(T|_{X_{2}})<1$, where $r(\cdot)$ refers to the spectral radius. The following example demonstrates  that the assumption on $X$ cannot be dropped and  an analogous result as Theorem \ref{thm:main} does not hold in the discrete-time case.

	\begin{example}
		Let $X=c_{0}(\mathbb{N})$ and let $(e_{n})_{n\in\NN}$ refer to the canonical basis. It is easy to see that the operator $T$ defined by $Ty=\sum_{m=1}^{\infty}(1-\frac{1}{m})y_{m}e_{m}$ for $y=\sum_{n=1}^{\infty}y_{n}e_{n}$, satisfies $\ell^{\infty}$-maximal regularity. Indeed, to see this, let $x_{k}=\sum_{m=1}^{\infty}x_{k,m}e_{m}\in X$ and consider, for $n\in\NN$,
		\begin{align*}
		\left\|\sum_{k=1}^{n}T^{n-k}(I-T)x_{k}\right\|={}&\left\|\sum_{k=1}^{n}\sum_{m=1}^{\infty}\left(1-\tfrac{1}{m}\right)^{n-k}\tfrac{1}{m}x_{k,m}e_{m}\right\|\\
								={}&\sup_{m}\,\tfrac{1}{m}\left|\sum_{k=1}^{n}\left(1-\tfrac{1}{m}\right)^{n-k}x_{k,m}\right|\\
								\leq{}& \sup_{k=1,..,n}\|x_{k}\|.
		\end{align*}
		On the other hand, it is obvious that there is no decomposition of $X$ into $T$-invariant closed subspaces $X_{1}$ and $X_{2}$ such that $T|_{X_{1}}=I$ and such that the spectral radius of $T|_{X_{2}}$ is less than $1$. 
	\end{example}
\section{An alternative proof for Theorem \ref{thm:main}}
We sketch an alternative argument for showing Theorem \ref{thm:main} which avoids the notion of $G_{\delta}$-embeddings and the fact that $Fav(\cS)=D(A)$ under the additional assumption that the dual semigroup $\cS^{*}$ is strongly continuous. The key is the following lemma which can be proved by carefully studying Baillon's space.
\begin{lemma}\label{lem41}
Let $\cS$ be a strongly continuous semigroup with unbounded generator $A$. Suppose that $A_{-1}$ is $\rL^{\infty}$-admissible and let $(z_{n})_{n\in\NN}$ be the elements spanning the Baillon space $Z$ from Proposition \ref{thm:Baillon}. Then $\sum_{n=1}^{\infty}z_{n}$ converges to an element $z\in X$ in the norm $\|R(\lambda,A)\cdot\|$, i.e.
\begin{equation}\label{eq:altproof1}\lim_{N\to\infty}R(\lambda,A)\sum_{n=1}^{N}z_{n}=R(\lambda,A)z
\end{equation}
 for $\lambda\in R(\lambda,A)$.
\end{lemma}

\begin{proof}[Alternative proof of Theorem 2.9, (2)$\implies$(3), if $\cS^{*}$ is strongly continuous]
As in the main proof of the theorem, we suppose by contradiction that $A$ is  unbounded.
Observe that \eqref{eq:altproof1} from Lemma \ref{lem41} implies that the series $\sum_{n=1}^{\infty}z_{n}$ converges to $z$ in the $\sigma(X,X^*)$-topology since $X^{*}=X^{\odot}$ by the assumption that $\cS^*$ is strongly continuous. Indeed, this follows easily from the fact that for all $f\in X^{*}$,
\[\langle \sum_{n=1}^{N}z_{n}, R(\lambda,A)^*f\rangle\to\langle R(\lambda,A)z,f\rangle,\qquad N\to\infty, \]
and the fact that $(\sum_{n=1}^{N}z_{n})_{N\in\NN}$ is bounded. Therefore, $z\in Z$ since $Z$ is weakly-closed as a norm-closed space. Finally, let $f_{n}:Z\to \mathbb{C}$ denote the coefficient functionals of the Schauder basis $(z_{n})_{n\in\NN}$ of $Z$, which can be continuously extended to functionals on $X$ by  the Hahn-Banach theorem. It follows that for fixed $n\in \NN$,
\[f_{n}\left(\sum_{m=1}^{N}z_{m}\right)=\langle \sum_{m=1}^{N}z_{m},f_{n}\rangle \to f_{n}(z),\qquad (N\to\infty),\]
which implies that $f_{n}(z)=1$ for all $n\in\mathbb{N}$. This contradicts $z\in Z$. 
\end{proof}
}

\section*{Acknowledgements} 
The authors would like to thank the anonymous referees for several helpful comments on the manuscript and in particular for pointing out to us the article \cite{RiFarw20}.


\end{document}